\documentclass[a4paper,12pt,dvipdfmx]{amsart}
\usepackage{amsmath}
\usepackage{amssymb}
\usepackage{enumerate}
\usepackage{amscd}
\usepackage{graphics}
\usepackage{latexsym}
\usepackage{verbatim} 
\usepackage{url}

\theoremstyle{plain}
\newtheorem{thm}{{\bf Theorem}}[section]
\newtheorem{cor}[thm]{{\bf  Corollary}}
\newtheorem{prop}[thm]{{\bf Proposition}}
\newtheorem{lemma}[thm]{{\bf Lemma}}
\newtheorem{fact}[thm]{{\bf Fact}}
\newtheorem{claim}[thm]{{\bf Claim}}

\theoremstyle{definition}
\newtheorem{define}[thm]{{\bf Definition}}
\newtheorem{note}[thm]{{\bf Note}}

\newtheorem{question}[thm]{{\bf Question}}

\newcommand{\cf}{\mathord{\mathrm{cf}}}

\newcommand{\size}[1]{\left\vert {#1} \right\vert}
\newcommand{\p}{\mathcal{P}}
\newcommand{\ot}{\mathord{\mathrm{ot}}}

\newcommand{\seq}[1]{\langle {#1} \rangle}

\newcommand{\ka}{\kappa}
\newcommand{\la}{\lambda}
\newcommand{\om}{\omega}

\newcommand{\pkl}{\mathcal{P}_\kappa \lambda}

\newcommand{\bbC}{\mathbb{C}}

\newcommand{\bbR}{\mathbb{R}}

\newcommand{\calA}{\mathcal{A}}

\newcommand{\calF}{\mathcal{F}}

\newcommand{\calU}{\mathcal{U}}

\title[]{A note on $\delta$-strongly compact cardinals}
\author[T. Usuba]{Toshimichi Usuba}
\address[T. Usuba]
{Faculty of Science and Engineering,
Waseda University, 
Okubo 3-4-1, Shinjyuku, Tokyo, 169-8555 Japan}
\email{usuba@waseda.jp}
\keywords{countably 
tight, $\delta$-strongly compact cardinal, Lindel\"of space, $\om_1$-strongly compact cardinal, uniform ultrafilter}
\subjclass[2010]{Primary 03E55, 54A25}

\begin{document}

\begin{abstract}
In this paper we investigate more characterizations and applications of 
$\delta$-strongly compact cardinals.
We show that, for a cardinal $\ka$, the following are equivalent:
(1) $\ka$ is $\delta$-strongly compact, (2) For every regular $\la \ge \ka$
there is a $\delta$-complete uniform ultrafilter over $\la$, and
(3) Every product space of $\delta$-Lindel\"of spaces is $\ka$-Lindel\"of.
We also prove that in the Cohen forcing extension,
the least $\om_1$-strongly compact cardinal
is an exact upper bound on 
the tightness of the products of two countably tight spaces.
\end{abstract}
\maketitle

\section{Introduction}
Bagaria and Magidor \cite{BM1, BM2} introduced the notion
of $\delta$-strongly compact cardinals, which is a variant of strongly compact cardinals.

\begin{define}[Bagaria-Magidor \cite{BM1, BM2}]
Let $\ka$, $\delta$ be uncountable cardinals with $\delta \le \ka$.
$\ka$ is \emph{$\delta$-strongly compact}
if for every set $A$, every $\ka$-complete filter over $A$
can be extended to a $\delta$-complete ultrafilter.
\end{define}

$\delta$-strongly compact cardinals, especially for the case $\delta=\om_1$, 
have various characterizations and many applications, 
see Bagaria-Magidor \cite{BM1, BM2}, Bagaria-da Silva \cite{BS}, and Usuba \cite{U1,U2}.
In this paper, we investigate  more characterizations and applications 
of $\delta$-strongly compact cardinals.

Ketonen \cite{K} characterized strongly compact cardinals by
the existence of uniform ultrafilters, where
a filter $F$ over a cardinal $\la$ is \emph{uniform} if
$\size{X}=\la$ for every $X \in F$.
Ketonen proved that an uncountable cardinal $\ka$ is strongly compact cardinal if,
and only if,
for every regular $\la \ge \ka$,
there exists a $\ka$-complete uniform ultrafilter over $\la$.
We prove a similar characterization for $\delta$-strongly compact cardinals.

\begin{thm}\label{prop3.3}
Let $\ka$ and $\delta$ be uncountable cardinals with $\delta \le \ka$.
Then $\ka$ is $\delta$-strongly compact
if, and only if, for every regular $\la \ge \ka$,
there exists a $\delta$-complete uniform ultrafilter over $\la$.
\end{thm}

In \cite{BM2}, Bagaria and Magidor 
characterized $\om_1$-strongly compact cardinals
in terms of topological spaces.
Let $\mu$ be a cardinal. A topological space $X$ is
\emph{$\mu$-Lindel\"of} if every open cover of $X$ has a 
subcover of size $<\mu$. 
An $\om_1$-Lindel\"of space is called a \emph{Lindel\"of space}.

Bagaria and Magidor 
proved that a cardinal $\ka$ is $\om_1$-strongly compact
if and only if every product space of Lindel\"of spaces is $\ka$-Lindel\"of.
Using Theorem \ref{prop3.3},
we generalize this result as follows:
\begin{thm}\label{thm3}
Let $\delta \le \ka$ be uncountable cardinals.
Then the following are equivalent:
\begin{enumerate}
\item $\ka$ is $\delta$-strongly compact.
\item For every family $\{X_i \mid i \in I\}$ of
$\delta$-Lindel\"of spaces,
the product space $\prod_{i \in I} X_i$ is $\ka$-Lindel\"of.
\end{enumerate}
\end{thm}

We turn to another topological property,
the tightness.
For a topological space $X$, the \emph{tightness number} $t(X)$ of $X$
is the minimum infinite cardinal $\ka$
such that whenever $A \subseteq X$ and $p \in \overline{A}$ (where $\overline{A}$ is the closure of $A$ in $X$),
there is $B \subseteq A$ with $\size{B} \le \ka$ and $p \in \overline{B}$.
If $t(X)=\om$, $X$ is called a \emph{countably tight space}.

The product of countably tight spaces need not to be countably tight:
A typical example is the sequential fan $S({\om_1})$.
It is a Fr\v echet-Urysohn space,
but the square of $S({\om_1})$ has uncountable tightness.
It is also known that if $\ka$ is a regular uncountable cardinal and
the set $\{\alpha<\ka \mid \cf(\alpha)=\om\}$ has a non-reflecting stationary subset,
then $t(S(\ka)^2) =\ka$ (see 
Eda-Gruenhage-Koszmider-Tamano-Todor\v cevi\'c \cite{EGKTT}).
In particular, under $V=L$, the tightness of the product of two Fr\v echet-Urysohn spaces
can be arbitrary large.

Among these facts, we show that an $\om_1$-strongly compact cardinal
gives an upper bound on the tightness of
the product of two countably tight spaces.

\begin{thm}\label{thm4}
If $\ka$ is $\om_1$-strongly compact,
then $t(X \times Y) \le \ka$ for every countably tight spaces $X$ and $Y$.
\end{thm}

We also show that an $\om_1$-strongly compact cardinal is an \emph{exact} upper bound in the Cohen forcing extension.
\begin{thm}\label{thm5}
Let $\bbC$ be the Cohen forcing notion, and $G$ be $(V, \bbC)$-generic.
Then for every cardinal $\ka$ 
the following are equivalent  in $V[G]$:
\begin{enumerate}
\item $\ka$ is $\om_1$-strongly compact.
\item For every countably tight spaces $X$ and $Y$
we have  $t(X \times Y) \le \ka$.
\item For every countably tight Tychonoff spaces $X$ and $Y$
we have  $t(X \times Y) \le \ka$.
\end{enumerate}
\end{thm}

Here we present some definitions and facts
which will be used later.
Throughout this paper, a filter means a proper filter.
\begin{define}
For an uncountable cardinal $\ka$ and a set $A$,
let $\p_\ka A=\{x \subseteq A \mid \size{x}<\ka\}$.
A filter $F$ over $\p_\ka A$ is \emph{fine}
if for every $a \in A$,
we have $\{x \in \p_\ka A\mid a \in x\} \in F$.
\end{define}
\begin{thm}[\cite{BM1, BM2}]
For  uncountable cardinals $\delta \le \ka$,
the following are equivalent:
\begin{enumerate}
\item $\ka$ is $\delta$-strongly compact.
\item For every cardinal $\la \ge \ka$,
there exists a $\delta$-complete fine ultrafilter over $\pkl$.
\item For every set $A$ with $\size{A} \ge \ka$,
there exists a $\delta$-complete fine ultrafilter over $\p_\ka A$.
\item For every cardinal $\la \ge \ka$,
there exists a definable elementary embedding $j: V \to M$ into some transitive model $M$
such that $\delta \le \mathrm{crit}(j) \le \ka$
and there is a set $A \in M$ with $\size{A}^{M}<j(\ka)$ and $j``\la \subseteq A$
(where $\mathrm{crit}(j)$ denotes the critical point of $j$).
\end{enumerate}  
\end{thm}

\begin{thm}[\cite{BM1, BM2}]
If $\ka$ is $\delta$-strongly compact,
then there is a measurable cardinal $\le \ka$.
\end{thm}

\section{On uniform ultrafilters}
In this section we give a proof of Theorem \ref{prop3.3}.
It can be obtained by a series of arguments in Ketonen \cite{K}
with some modifications.

\begin{fact}
Let $j:V \to M$ be a definable elementary embedding into some transitive model $M$.
Let $A$ be a non-empty set,
and let $s \in j(A)$.
Define the family $U \subseteq \p(A)$ as follows. For $X \subseteq A$,
\begin{align*}
X \in U \text{ if and only if } s \in j(X).
\end{align*}
Then $U$ is a $\mathrm{crit}(j)$-complete (possibly principal) ultrafilter over $A$.
Moreover, if $A$ is a regular uncountable cardinal $\la$ and 
$s=\sup(j``\la) \in j(\la)$,
then $U$ is a uniform ultrafilter.
\end{fact}

\begin{lemma}\label{2.2+}
Suppose $\ka$ is $\delta$-strongly compact for some uncountable $\delta \le \ka$.
Then for every regular $\la \ge \ka$
there exists a $\delta$-complete uniform ultrafilter over $\la$.
\end{lemma}
\begin{proof}
Fix a regular $\la \ge \ka$,
and take an elementary embedding $j:V \to M$
such that $\delta \le \mathrm{crit}(j)\le \ka$, and
there is $A \in M$ with $j``\la \subseteq A \subseteq j(\la)$ and
$\size{A}^M<j(\ka)$.
Then we have $\sup(j``\la)<j(\la)$ because $j(\la)$ is regular in $M$ and $j(\la)>j(\ka)$.
Now define the ultrafilter $U$ over $\la$ by, for $X \subseteq \la$, 
\[
X \in U \text{ if and only if } \sup(j``\la) \in j(X).
\]
$U$ is a $\delta$-complete uniform ultrafilter over $\la$.
\end{proof}
For the converse direction, we need several definitions and lemmas.

Let $U$ be a $\sigma$-complete ultrafilter over some set $A$.
Let $\mathrm{Ult}(V, M)$ denote the ultrapower of $V$ by $U$,
and identify the ultrapower with its transitive collapse $M$.
Let $j:V \to M \approx \mathrm{Ult}(V, U)$ be the elementary embedding induced by $U$.
It is known that  $\mathrm{crit}(j) \ge \delta$ if and only if $U$ is $\delta$-complete.
Let $id_A$ denote the identity map on $A$,
and for a function $f$ on $A$, let $[f]_U \in M$ denote the equivalence class of $f$ modulo $U$.
We know $[f]_U=j(f)([id_A]_U)$ because
\begin{align*}
[g]_U \in [f]_U & \iff \{x \in A \mid g(x) \in f(x)\} \in U\\
& \iff \{x \in A \mid g(x) \in f(id_A(x))\} \in U\\
& \iff [g]_U \in j(f)([id_A]_U).
\end{align*}

\begin{define}
Let $\mu$, $\nu$ be cardinals with $\mu \le \nu$.
An ultrafilter $U$ over some set $A$ is said to be \emph{$(\mu, \nu)$-regular}
if there is a family $\{X_\alpha \mid \alpha<\nu\}$ of sets in $U$
such that for every $a \in [\nu]^\mu$,
we have $\bigcap_{\alpha \in a} X_\alpha=\emptyset$.
\end{define}
We note that if $\nu$ is regular and $U$ is $(\mu, \nu)$-regular,
then $\size{X} \ge \nu$ for every $X \in U$.

\begin{lemma}\label{5.4}
Let $\mu \le \nu$ be cardinals where $\nu$ is regular.
Let $U$ be a $\sigma$-complete ultrafilter over some set $A$, and 
$j: V \to M \approx \mathrm{Ult}(V, U)$  the elementary embedding induced by $U$.
Then $U$ is $(\mu, \nu)$-regular if and only if $\cf^M(\sup(j``\nu))<j(\mu)$.
\end{lemma}
\begin{proof}

First suppose $U$ is $(\mu, \nu)$-regular, and
let $\{X_\alpha \mid \alpha<\nu\}$ be a witness.
Let $j(\{X_\alpha \mid \alpha<\nu\})=
\{Y_\alpha \mid \alpha<j(\nu)\}$.
Let $a=\{\alpha<\sup(j``\nu) \mid [id_A]_U \in Y_\alpha\} \in M$.
For $\alpha<\nu$, 
since $X_\alpha \in U$, we have $[id_A]_U \in j(X_\alpha)=Y_{j(\alpha)}$.
Thus we know $j``\nu \subseteq a$.
Hence $a$ is unbounded in $\sup(j``\nu)$,
and therefore $\cf^M(\sup(j``\nu))\le \size{a}^M$.
By the choice of $a$, we have $\bigcap_{\alpha \in a} Y_\alpha \neq \emptyset$.
Since $j(U)$ is $(j(\mu), j(\nu))$-regular in $M$ witnessed by the family $\{Y_\alpha \mid \alpha<j(\nu)\}$,
we have $\size{a}^M<j(\mu)$, and $\cf^M(\sup(j``\nu))<j(\mu)$.

For the converse, suppose $\cf^M(\sup(j``\nu))<j(\mu)$.
Take a function $f:A \to \nu+1$ such that
$[f]_U=j(f)([id_A]_U)=\sup(j``\nu)$ in $M$.
Then $Z=\{x \in A \mid \cf(f(x))<\mu\} \in U$.
For each $x \in Z$, take $c_x \subseteq f(x)$ such that
$\ot(c_x)=\cf(f(x))$ and $\sup(c_x)=f(x)$.
\begin{claim}
There is  a strictly increasing sequence $\seq{\nu_i \mid i<\nu}$ in $\nu$
such that
$\{x\in Z \mid [\nu_i, \nu_{i+1}) \cap c_x \neq \emptyset\} \in U$ for every $i<\nu$.
\end{claim}
\begin{proof}
By induction on $i<\nu$.
Suppose $k<\nu$ and $\nu_i$ is defined for all $i<k$.
If $k$ is limit, since $\nu$ is regular, we have $\sup\{\nu_i \mid i<k\}<\nu$.
Then take $\nu_k<\la$ with 
$\sup\{\nu_i \mid i<k\}<\nu_k$.
Suppose $k$ is successor, say $k=l+1$. 
Consider the set $c_{[id_A]_U} \subseteq j(f)([id_A]_U)=\sup(j``\nu)$.
$c_{[id_A]_U}$ is unbounded in $\sup(j``\nu)$. 
Then pick some $\xi \in c_{[id_A]_U}$ with $j(\nu_l)<\xi$,
and take $\nu_k<\nu$ with $\xi<j(\nu_k)$.
We have $ \xi \in [j(\nu_l), j(\nu_k)) \cap c_{[id_A]_U}$,
hence $\{x \in Z \mid [\nu_l, \nu_k) \cap c_x \neq \emptyset\} \in U$.
\end{proof}

Finally, let $X_i=\{x \in Z \mid 
 [\nu_i, \nu_{i+1}) \cap c_x \neq \emptyset\}$,
which is in $U$ by the choice of the $\nu_i$'s.
We check that $\{X_i \mid i<\nu\}$ witnesses that $U$ is $(\mu, \nu)$-regular.
So take $a \in [\nu]^\mu$, and suppose to the contrary that $x \in \bigcap_{i \in a} X_i$.
Then $[\nu_i, \nu_{i+1}) \cap c_x \neq \emptyset$ for every $i \in a$.
Since $\seq{\nu_i \mid i<\nu}$ is strictly increasing,
we have $\size{c_x} \ge \mu$, this contradicts the choice of $c_x$.
\end{proof}

\begin{lemma}\label{5.6}
Let $\ka$ and $\delta$ be uncountable cardinals with $\delta \le \ka$.
Then the following are equivalent:
\begin{enumerate}
\item $\ka$ is $\delta$-strongly compact.
\item For every regular $\la \ge \ka$,
there exists a $\delta$-complete $(\ka, \la)$-regular ultrafilter over some set $A$.
\end{enumerate}
\end{lemma}
\begin{proof}
Suppose $\ka$ is $\delta$-strongly compact.
Fix a regular cardinal $\la \ge \ka$,
and take a $\delta$-complete fine ultrafilter $U$ over $\pkl$.
For $\alpha<\la$, let $X_\alpha=\{x \in \pkl \mid \alpha \in x\} \in U$.
Then the family $\{X_\alpha \mid \alpha<\la\}$ witnesses that
$U$ is $(\ka, \la)$-regular.

For the converse,
pick a cardinal $\la \ge \ka$. 
By (2), there is a $\delta$-complete $(\ka, \la^+)$-regular ultrafilter $W$ over some set $A$.
Take the elementary embedding $i:V \to N \approx \mathrm{Ult}(V, W)$.
We have $\cf^N(\sup(i``\la^+))<i(\ka)$ by Lemma \ref{5.4}.
By the  elementarity of $i$, 
one can check that for every stationary $S \subseteq \{\alpha<\la^+ \mid \cf(\alpha)=\om\}$,
we have that $i(S) \cap \sup(i``\la^+)$ is stationary in $\sup(i``\la^+)$ in $N$ (e.g., see \cite{BM2}).
Fix a stationary partition $\{S_i \mid i<\la\}$ of 
$\{\alpha<\la^+ \mid \cf(\alpha)=\om\}$,
and let $i(\{S_i \mid i<\la\})=\{S'_\alpha \mid \alpha<i(\la)\}$.
Let $a=\{\alpha \in i(\la) \mid S'_\alpha \cap \sup(i``\la^+)$ is stationary in $\sup(i``\la^+)$ in $N\}$.
We have $a \in N$ and $i``\la \subseteq a$.
Moreover, since $\cf^N(\sup(i``\la^+)) <i(\ka)$,
we have $\size{a}^N<i(\ka)$.
Hence $a \in i(\p_\ka \la)$,
and then define the filter $U$ over $\p_\ka \la$ as follows.
For $X \subseteq \pkl$,
\[
X \in U \text{ if and only if } a  \in i(X).
\]
$U$ is a $\delta$-complete ultrafilter.
We see that $U$ is fine, so fix $\alpha<\la$. 
We know $j(\alpha) \in j``\la \subseteq a$,
thus $a \in \{x \in j(\pkl) \mid j(\alpha) \in x \}=j(\{x \in \pkl \mid \alpha \in x\})$,
and therefore $\{x \in \pkl \mid \alpha \in x\} \in U$.
\end{proof}

\begin{define}
Let $\la$ be an uncountable cardinal and $U$ an ultrafilter over $\la$.
$U$ is \emph{weakly normal} if for every $f:\la \to \la$
with $\{\alpha<\la \mid f(\alpha)<\alpha\} \in U$, there is
$\gamma<\la$ such that $\{\alpha<\la \mid f(\alpha)<\gamma\} \in U$.
\end{define}

\begin{lemma}\label{5.3}
Let $\la$ be a regular cardinal, and $\delta \le \la$ an uncountable cardinal.
If $\la$ carries a $\delta$-complete uniform ultrafilter,
then $\la$ carries a $\delta$-complete weakly normal uniform ultrafilter as well.
\end{lemma}
\begin{proof}
Let $U$ be a 
$\delta$-complete uniform ultrafilter over $\la$,
and $j: V \to M \approx \mathrm{Ult}(V, U)$ be the elementary embedding induced by $U$.
Since $U$ is uniform, we have $\sup(j``\la)\le [id_\la]_U<j(\la)$.
Then define the filter $W$ over $\la$ as follows. For every $X \subseteq \la$,
\[
X \in W \text{ if and only if } \sup(j``\la) \in j(X).
\]
$W$ is a $\delta$-complete uniform ultrafilter over $\la$.
We have to check that $W$ is weakly normal.
Take $f:\la \to \la$ and suppose $X=\{\alpha<\la \mid f(\alpha) <\la\} \in U$.
We know $\sup(j``\la) \in j(X)$, hence $j(f)(\sup j``\la)<\sup(j``\la)$.
Then we can pick $\gamma<\la$ with $j(f)(\sup(j``\la)) < j(\gamma)$.
By the definition of $W$, we have that $\{\alpha<\la \mid f(\alpha) <\gamma\} \in W$.
\end{proof}

The following is an immediate consequence of Lemma \ref{5.4} and the weak normality:
\begin{lemma}\label{2.7}
Let $\la$ be a regular cardinal, 
and $U$ a $\sigma$-complete weakly normal ultrafilter over $\la$.
Let $j: V \to M \approx \mathrm{Ult}(V, U)$ be the elementary embedding induced by $U$.
Then $[id_\la]_U=\sup(j``\la)$.
Hence $U$ is $(\mu, \la)$-regular if and only if $\{\alpha <\la \mid \cf(\alpha)<\mu\} \in U$.
\end{lemma}

\begin{define}
Let $A$ be a non-empty set,
and $U$ an ultrafilter over $A$.
Let $X \in U$, and for each $x \in X$,
let $W_x$ be an ultrafilter over some set $A_x$.
Then the \emph{$U$-sum} of $\{W_x \mid x \in X\}$
is the collection $D$ of subsets of $\{\seq{x,y } \mid x \in X, y \in A_x\}$
such that for every such subset $Y$,
\[
Y \in D \text{ if and only if } \{x \in X \mid \{y\in A_x \mid \seq{x,y} \in Y\} \in W_x\}
\in U.
\]
\end{define}
\begin{fact}
The $U$-sum $D$ is an ultrafilter over the set $\{\seq{x,y} \mid x \in X, y \in A_x \}$,
and if $U$ and the $W_x$'s are $\delta$-complete,
then so is $D$.
\end{fact}

\begin{lemma}
Let $\ka$ and $\delta$ be uncountable cardinals with $\delta \le \ka$.
Suppose that for every regular $\la \ge \ka$,
there exists a $\delta$-complete uniform ultrafilter over $\la$.
Then $\ka$ is $\delta$-strongly compact.
\end{lemma}
\begin{proof}
First suppose $\ka$ is regular.
To show that $\ka$ is $\delta$-strongly compact cardinal,
by Lemma \ref{5.6}
it is enough to see that for every regular $\la \ge \ka$,
there exists a $\delta$-complete $(\ka, \la)$-regular ultrafilter over $\la$.
We prove this by induction on $\la$.
For the base step $\la=\ka$,
note that $\{\alpha<\ka \mid \cf(\alpha)<\ka\}=\ka$.
By Lemma \ref{5.3}, we can take a $\delta$-complete weakly normal uniform ultrafilter $U$ over $\ka$.
Then $\{\alpha<\ka \mid \cf(\alpha)<\ka\} =\ka \in U$,
hence $U$ is $(\ka, \ka)$-regular by Lemma \ref{2.7}.

Let $\la>\ka$ be regular, and suppose for every regular $\nu$ with $\ka \le \nu<\la$,
there exists a $\delta$-complete $(\ka, \nu)$-regular ultrafilter $U_\nu$ over $\nu$.
Fix a $\delta$-complete weakly normal uniform ultrafilter $U$ over $\la$.
If $\{\alpha <\la \mid \cf(\alpha)<\ka\} \in U$,
then $U$ is $(\ka, \la)$-regular by Lemmas \ref{5.4} and \ref{2.7}, and we are done.
So, suppose $X^*=\{\alpha<\la \mid \cf(\alpha) \ge \ka \} \in U$.
For each $\alpha \in X^*$, let $W_\alpha$ be the $\delta$-complete $(\ka, \cf(\alpha))$-regular ultrafilter $U_{\cf(\alpha)}$ over $\cf(\alpha)$.
Let $B=\{\seq{\alpha,\beta} \mid \alpha \in X^*, \beta<\cf(\alpha)\}$.
Note that $\size{B}=\la$.
Let us consider the $U$-sum $D$ of $\{W_\alpha \mid \alpha \in X^*\}$.
$D$ is a $\delta$-complete ultrafilter over $B$.
We claim that $D$ is $(\ka, \la)$-regular,
and then we can easily take a $\delta$-complete  $(\ka, \la)$-regular ultrafilter over $\la$.

For $\alpha \in X^*$, let $j_\alpha:V \to M_\alpha \approx \mathrm{Ult}(V, W_\alpha)$
be the elementary embedding induced by $W_\alpha$.
Let $g_\alpha :\cf(\alpha) \to \alpha+1$ be a 
function which represents $\sup(j_\alpha``\alpha)$ in $M_\alpha$.
Note that, since $W_\alpha$ is $(\ka, \cf(\alpha))$-regular,
we have $\cf^{M_\alpha}(\sup(j_\alpha``\alpha))=
\cf^{M_\alpha}(\sup(j_\alpha``\cf(\alpha)))
<j_\alpha(\ka)$,
so $\{\beta<\cf(\alpha) \mid \cf(g_\alpha(\beta))<\ka\} \in W_\alpha$.

Let $i:V \to N \approx \mathrm{Ult}(V, D)$ be the elementary embedding
induced by $D$.
Define the function $g$ on $B$
by $g(\alpha,\beta)=g_{\alpha}(\beta)$.
\begin{claim}
$\sup(i``\la)=[g]_D$.
\end{claim}
\begin{proof}
First, for $\gamma<\la$,
we have $X^* \setminus \gamma \in U$,
and $\{\beta<\cf(\alpha) \mid g_\alpha(\beta) \ge \gamma\} \in W_\alpha$ for all $\alpha \in X^*\setminus \gamma$.
This means  that $\{\seq{\alpha, \beta} \in B \mid g(\alpha,\beta) \ge \gamma\} \in D$,
and $i(\gamma) \le [g]_D$.
Next, take a function $h$ on $B$ with $[h]_D <[g]_D$.
Then $\{\seq{\alpha,\beta} \in B \mid h(\alpha,\beta)<g(\alpha,\beta)\} \in D$,
and $X'=\{\alpha \in X^* \mid \{\beta<\cf(\alpha) \mid h(\alpha,\beta)<g(\alpha,\beta)\} \in W_\alpha\} \in U$.
For $\alpha \in X'$, we know $\{\beta<\cf(\alpha) \mid h(\alpha,\beta)<g(\alpha,\beta)\} \in W_\alpha$.
Because $g(\alpha,\beta)=g_\alpha(\beta)$ represents $\sup(j_\alpha``\alpha)$,
there is some $\gamma_\alpha<\alpha$
such that $\{\beta<\cf(\alpha) \mid h(\alpha,\beta)<\gamma_\alpha\} \in W_\alpha$.
Now, since $U$ is weakly normal and $\gamma_\alpha<\alpha$ for $\alpha \in X'$,
there is some $\gamma<\la$ such that
$\{\alpha \in X' \mid \gamma_\alpha <\gamma\} \in U$.
Then we have $[h]_D<i(\gamma)<\sup(i``\lambda)$.
\end{proof}
Finally, since $\{\beta<\cf(\alpha) \mid \cf(g(\alpha,\beta))<\ka \} \in W_\alpha$
for every $\alpha \in X^*$,
we have $\{\seq{\alpha,\beta} \in B \mid \cf(g(\alpha,\beta))<\ka \} \in D$,
this means that $\cf^N([g]_D)=\cf^N(\sup(i``\la))<i(\ka)$,
 and $D$ is $(\ka, \la)$-regular by Lemma \ref{5.4}.

Next suppose $\ka$ is singular.
By induction on $\la \ge \ka$, 
we take a $\delta$-complete $(\ka,\la)$-regular ultrafilter 
over $\la$ for regular $\la$.
For the base step $\la=\ka^+$, by Lemma \ref{5.3} take a $\delta$-complete weakly normal uniform ultrafilter $U$ over $\ka^+$.
We have $\ka^+=\{\alpha <\ka^+ \mid \cf(\alpha) \le \ka\} \in U$,
and in fact $\{\alpha<\ka^+ \mid \cf(\alpha)<\ka\} \in U$ since $\ka$ is singular.
Then $U$ is $(\ka, \ka^+)$-regular.
The rest is the same to the case that $\ka$ is regular.
\end{proof}
This completes the proof of Theorem \ref{prop3.3}.

Using Theorem \ref{prop3.3}, we also have the following characterization of $\delta$-strongly compact
cardinals.
\begin{cor}\label{2.10}
Let $\delta \le \ka$ be uncountable cardinals.
Then the following are equivalent:
\begin{enumerate}
\item $\ka$ is $\delta$-strongly compact.
\item For every regular $\la \ge \ka$,
there is a definable elementary embedding $j:V \to M$ into some transitive model $M$
with $\delta \le \mathrm{crit}(j) \le  \ka$ and $\sup(j``\la)<j(\la)$.
\item For every regular $\la \ge \ka$,
there is a definable elementary embedding $j:V \to M$ into some transitive model $M$
with $\delta \le \mathrm{crit}(j)$ and $\sup(j``\la)<j(\la)$.
\end{enumerate}
\end{cor}
\begin{proof}
For (1) $\Rightarrow$ (2), suppose $\ka$ is $\delta$-strongly compact.
Then for every regular $\la \ge \ka$,
there is a $\delta$-complete fine ultrafilter over $\pkl$.
If $j:V \to M$ is the ultrapower induced by the ultrafilter,
then we have that the critical point of $j$ is between $\delta$ and $\ka$,
and $\sup(j``\la)<j(\la)$.

(2) $\Rightarrow$ (3) is trivial.
For (3) $\Rightarrow$ (1), 
it is enough to see that every regular $\la \ge \ka$ carries a $\delta$-complete uniform ultrafilter.
Let $\la \ge \ka$ be regular.
Take an elementary embedding $j:V \to M$ with $\delta \le \mathrm{crit}(j)$ and $\sup(j``\la)<j(\la)$.
Define $U \subseteq \p(\la)$ as follows. For $X \subseteq \la$,
\[
X\in U \text{ if and only if } \sup(j``\la) \in j(X).
\]
We know that $U$ is a $\delta$-complete uniform ultrafilter over $\la$.
\end{proof}

Bagaria and Magidor \cite{BM2} proved that
the least $\delta$-strongly compact cardinal must be a limit cardinal.
We can prove the following slightly stronger result using Theorem \ref{prop3.3}.

For a regular cardinal $\nu$ and $f, g \in {}^\nu \nu$,
define $f \le^* g$ if the set $\{\xi<\nu \mid f(\xi) >g(\xi)\}$ is bounded in $\nu$.
A family $F \subseteq {}^\nu \nu$ is \emph{unbounded} if
there is no $g \in {}^\nu \nu$ such that  $f\le^*g$ for every $f \in F$.
Then let $\mathfrak{b}_\nu=\min\{\size{F} \mid F \subseteq {}^\nu \nu$ is unbounded$\}$.
Note that $\mathfrak b_\nu$ is regular and 
$\nu^+ \le \mathfrak b_\nu \le 2^\nu$.

\begin{prop}
Let $\delta$ be an uncountable cardinal,
and suppose $\ka$ is the least $\delta$-strongly compact cardinal.
Then for every cardinal $\mu<\ka$, there is a regular $\nu$ with
$\mu \le \nu<\mathfrak b_\nu<\ka$.
As an immediate consequence, $\ka$ is a limit cardinal.
\end{prop}
\begin{proof} 
Fix $\mu<\ka$. 
Take a regular $\nu$ as follows.
If $\mu \ge \delta$, by the minimality of $\ka$,
there is the least regular $\nu \ge \mu$
such that $\nu$ cannot carry a $\delta$-complete uniform ultrafilter over $\nu$.
We know $\nu <\ka$ since $\ka$ is $\delta$-strongly compact.
If $\mu <\delta$, let $\nu=\mu^+$. Thus, $\nu$ is regular with $\nu \le \delta \le \ka$.
We show that $\mathfrak b_\nu<\ka$ in both cases.
Let $\la=\mathfrak b_\nu$, and suppose to the contrary that $\la \ge \ka$.
By Corollary \ref{2.10},
we can take an elementary embedding $j:V \to M$ with $\delta \le \mathrm{crit}(j) \le \ka$
and $\sup(j``\la)<j(\la)$.
Then we have $\sup(j``\nu)=j(\nu)$; 
If $\mu<\delta$,
then $\nu=\mu^+ \le \delta \le \mathrm{crit}(j)$.
Since $\mathrm{crit}(j)$ is measurable and $\nu$ is successor,
we have $\nu<\mathrm{crit}(j)$. Hence $j(\nu)=\nu=\sup(\nu)=\sup(j``\nu)$.
If $\mu \ge \delta$ but $\sup(j``\nu)<j(\nu)$,
we can take a $\delta$-complete uniform ultrafilter 
$U=\{X \subseteq \nu \mid \sup(j``\nu) \in j(X)\}$ over $\nu$.
This contradicts  the choice of $\nu$.

Fix an unbounded set $F \subseteq {}^\nu \nu$ with size $\la$.
Let $F=\{f_\alpha \mid \alpha<\la \}$.
Consider $j(F)=\{f'_\alpha \mid  \alpha<j(\la)\}$.
Let $\gamma=\sup(j``\la)<j(\la)$.
By the elementarity of $j$,
the set $\{f'_\alpha \mid \alpha<\gamma\}$ is bounded in $j({}^\nu \nu)$ in $M$.
Thus there is $g' \in j({}^\nu \nu)$ such that $f'_\alpha \le^* g'$ for every $\alpha<\gamma$.
Define $g \in {}^\nu \nu$ by
\[
g(\xi)=\min\{\eta<\nu \mid g'(j(\xi)) \le j(\eta)\}
\]
for every $\xi<\nu$, this is well-defined since $\sup(j``\nu)=j(\nu)$.
Then we have $g'(j(\xi)) \le j(g(\xi))$ for every $\xi<\nu$.
Since the family $\{f_\alpha \mid \alpha<\la\}$ is unbounded,
there is $\alpha<\la$ with $f_\alpha \not \le^* g$.
$j(f_\alpha)=f'_{j(\alpha)} \le^* g'$,
thus there is $\eta'<j(\nu)$ such that
$j(f_\alpha)(\xi)\le g'(\xi)$ for every $\xi$ with $\eta'<\xi<j(\nu)$.
Since $j(\nu)=\sup(j``\nu)$,
we can take $\eta<\nu$ with $\eta' <j(\eta)$.
For $\xi$ with $\eta<\xi<\nu$,
we have $\eta'<j(\eta)<j(\xi)<j(\nu)$.
Thus $j(f_\alpha(\xi))=f'_{j(\alpha)}(j(\xi)) \le g'(j(\xi))$.
Hence $j(f_\alpha(\xi)) \le g'(j(\xi)) \le j(g(\xi))$ for every $\xi>\eta$.
By the elementarity of $j$,
we have that $f_\alpha (\xi) \le g(\xi)$ for every $\xi > \eta$,
 therefore $f_\alpha \le^* g$.
This is a contradiction.
\end{proof}

\begin{question}
For an uncountable cardinal $\delta$,
is the least $\delta$-strongly compact cardinal a strong limit?
Or a fixed point of the $\aleph$ or $\beth$-functions?
\end{question}

\section{On Products of $\delta$-Lindel\"of spaces}

In this section we give a proof of Theorem \ref{thm3}.
The direction (1) $\Rightarrow$ (2) just follows from the same proof as in \cite{BM2}.
For the converse direction in the case $\delta=\om_1$,
in \cite{BM2}, they used an algebraic method.
We give a direct proof, an idea that comes from Gorelic \cite{G}.

Suppose $\ka$ is not $\delta$-strongly compact.
By Theorem \ref{prop3.3}, there is a regular cardinal $\la \ge \ka$
such that $\la$ cannot carry a $\delta$-complete uniform ultrafilter.
Let $\calF$ be the family of all partitions of $\la$ with size $<\delta$,
that is, each $\calA \in \calF$ is a family of pairwise disjoint subsets of $\la$
with $\bigcup \calA=\la$ and $\size{\calA}<\delta$.
Let $\{\calA^\alpha \mid \alpha<2^\la\}$ be an enumeration of $\calF$.
For $\alpha<2^\la$, let $\delta_\alpha=\size{\calA^\alpha}<\delta$,
and $\{A^\alpha_\xi \mid \xi<\delta_\alpha\}$  be an enumeration of $\calA^\alpha$.
We identify $\delta_\alpha$ as a discrete space, which is trivially $\delta$-Lindel\"of.
We show that the product space $X=\prod_{\alpha<2^\la} \delta_\alpha$ is not $\ka$-Lindel\"of.

For $\gamma<\la$,
define $f_\gamma \in X$ as follows.
For $\alpha<2^\la$, since $\calA^\alpha$ is a partition of $\la$,
there is a unique $\xi<\delta_\alpha$ with $\gamma \in A^\alpha_\xi$.
Then let $f_\gamma(\alpha)=\xi$.

Let $Y=\{f_\gamma \mid \gamma<\la\}$.
It is clear that $\size{Y}=\la$.
\begin{claim}
For every $g \in X$, there is an open neighborhood $O$ of $g$ such that
 $\size{O \cap Y}<\la$.
\end{claim}
\begin{proof}
Suppose not.
Then the family $\{A^\alpha_{g(\alpha)} \mid \alpha<2^\la\}$ has the finite intersection property,
moreover for every finitely many $\alpha_0,\dotsc, \alpha_n<2^\la$,
the intersection $\bigcap_{i \le n} A^{\alpha_i}_{g(\alpha_i)}$ has cardinality $\la$.
Hence we can find a uniform ultrafilter $U$ over $\la$ extending 
$\{A^\alpha_{g(\alpha)} \mid \alpha<2^\la\}$.
By our assumption, $U$ is not $\delta$-complete.
Then we can take a partition $\calA$ of $\la$ with size $<\delta$
such that $A \notin U$ for every $A \in \calA$.
We can take $\beta<2^\la$ with $\calA=\calA^\beta$.
However then $A^\beta_{g(\beta)} \in \calA$ and $A^\beta_{g(\beta)} \in U$,
this is a contradiction.
\end{proof}

For each $g \in X$, take an open neighborhood $O_g$ of $g$ with
$\size{O_g \cap Y}<\la$.
Let $\calU=\{O_g \mid g \in X\}$.
$\calU$ is an open cover of $X$,
but has no subcover of size $<\la$ because $\size{Y}=\la$.
Hence $\calU$ witnesses that $X$ is not $\la$-Lindel\"of,
and not $\ka$-Lindel\"of. This completes our proof.
\\

By the same proof, we have:
\begin{cor}
Let $\ka$ be an uncountable cardinal, and $\delta<\ka$ a cardinal.
Then the following are equivalent:
\begin{enumerate}
\item $\ka$ is $\delta^+$-strongly compact.
\item Identifying $\delta$ as a discrete space,
for every cardinal $\la$, the product space $\delta^\la$ is $\ka$-Lindel\"of.
\end{enumerate}
\end{cor}

\section{On  products of countably tight spaces}

We prove Theorems \ref{thm4} and \ref{thm5} in this section.
For a topological space $X$ and $Y \subseteq X$,
let $\overline{Y}$ denote the closure of $Y$ in $X$.

\begin{lemma}\label{4.1}
Let $S$ be an uncountable set and $U$ a $\sigma$-complete ultrafilter over $S$.
Let $X$ be a countably tight space,
and $\{O_s \mid s \in S\}$  a family of open sets in $X$.
Define the set $O \subseteq X$ by
$x \in O  \iff \{s \in S \mid x \in O_s\} \in U$
for $x \in X$.
Then $O$ is open in $X$.
\end{lemma}
\begin{proof}
It is enough to show that $\overline{X \setminus O} \subseteq X \setminus O$.
Take $x \in \overline{X \setminus O}$, and suppose to the contrary that
$x \notin X \setminus O$. We have $\{s \in S \mid x \in O_s\} \in U$.
Since $X$ is countably tight, there is a countable $A \subseteq X \setminus O$ with
$x \in \overline{A}$.
For each $y \in A$, we have $\{s \in S \mid y \notin O_s\} \in U$.
Since $A$ is countable and $U$ is $\sigma$-complete,
there is $s \in S$ such that $y \notin O_s$ for every $y \in A$ 
but $x \in O_s$. Then $A \subseteq X \setminus O_s$.
Since $O_s$ is open, we have $\overline{X \setminus O_s} \subseteq X \setminus O_s$.
Hence $x \in \overline{A} \subseteq \overline{X \setminus O_s} \subseteq X \setminus O_s$,
and $x \in O_s$. This is a contradiction.
\end{proof}

The following proposition immediately yields Theorem \ref{thm4}.

\begin{prop}\label{prop4.2}
Suppose $\ka$ is $\om_1$-strongly compact, and
$\mu \le \ka$ the least measurable cardinal.
Let $I$ be a set with $\size{I}<\mu$,
and $\{X_i \mid i \in I\}$ a family of
countably tight spaces.
Then  $t(\prod_{i \in I}X_i) \le \ka$.
More precisely, for every $A \subseteq \prod_{i \in I} X_i$ and
$f \in \overline{A}$, there is $B \subseteq A$ such that
$\size{B}<\ka$ and $f \in \overline{B}$.
\end{prop}
\begin{proof}
Take $A \subseteq \prod_{i \in I}X_i$ and $f \in \overline{A}$.
We will find $B \subseteq A$ with $\size{B}<\ka$ and $f \in \overline{B}$.

Since $\ka$ is $\om_1$-strongly compact,
we can find a $\sigma$-complete fine ultrafilter $U$ over $\p_\ka (\prod_{i \in I}X_i)$.
Note that $U$ is in fact $\mu$-complete.
We show that $\{s \in \p_\ka (\prod_{i \in I}X_i) \mid f \in \overline{A \cap s}\} \in U$.
Suppose not and let $E =\{s \in \p_\ka (\prod_{i \in I}X_i)\mid
f \notin \overline{A \cap s}\} \in U$.
For each $s \in E$, since $f \notin \overline{A \cap s}$,
we can choose finitely many $i_0^s, \dotsc, i_n^s \in I$ and 
open sets $O_{i_k}^s \subseteq X_{i_k}$ respectively
such that $f(i_k^s) \in O_{i_k}^s$ for every $k \le n$
but $\{g \in A\cap s \mid \forall k \le n\,(g(i_k^s) \in O^s_{i_k})\}=\emptyset$.
Since $U$ is $\mu$-complete and $\size{I}<\mu$,
we can find $i_0,\dotsc, i_n \in I$
such that $E'=\{s \in E  \mid \forall k \le n\,(i_{k}^s=i_k)\} \in U$.

For each $i_k$,
let $O_{i_k} \subseteq X_{i_k}$ be the set  defined by
$x \in O_{i_k} \iff \{s \in E' \mid x \in O_{i_k}^s\} \in U$.
By Lemma \ref{4.1}, each $O_{i_k}$ is open in $X_{i_k}$ with
$f(i_k) \in O_{i_k}$.
Since $f \in \overline{A}$,
there is $h \in A$ such that
$h(i_k) \in O_{i_k}$ for every $k \le n$.
Because $U$ is fine, we can take $s \in E'$ with
$h \in A \cap s$ and $h(i_k) \in O^s_{i_k}$ for every $k \le n$.
Then $h \in \{g \in A \cap s \mid \forall k \le n\,(g(i_k) \in O^s_{i_k})\}$,
which is a contradiction.
\end{proof}

\begin{note}
\begin{enumerate}
\item 
The restriction ``$\size{I}<\mu$'' in Proposition \ref{prop4.2} cannot be eliminated.
If $I$ is an infinite set and $\{X_i \mid i \in I\}$ is a family of $T_1$ spaces with $\size{X_i}\ge 2$,
then $t(\prod_{i \in I} X_i) \ge \size{I}$:
For each $i \in X$ take distinct points $x_i, y_i \in X$.
For each finite subset $a \subseteq I$,
define $f_a \in \prod_{i \in I} X_i $ by  $f_a(i)=x_i$ if $i \in a$,
and $f_a(i) =y_i$ otherwise.
Let $X=\{f_a \mid a \in [I]^{<\om}\}$, and $g$ be the function with $g(i)=x_i$ for $i \in I$.
Then $g \in \overline{X}$ but for every $Y \subseteq X$ with $\size{Y}<\size{I}$ we have $g \notin \overline{Y}$.
\item We do not know 
if Proposition \ref{prop4.2} can be improved as follows:
If $\ka$ is the least $\om_1$-strongly compact
and $I$ is a set with size $<\ka$,
then the product of countably tight spaces indexed by $I$ has tightness $\le \ka$.
\end{enumerate}
\end{note}

Recall that the Cohen forcing notion $\bbC$ is the poset $2^{<\om}$ with
the reverse inclusion order.
\begin{prop}\label{4.4}
Let $\ka$ be a cardinal which is not $\om_1$-strongly compact.
Let $\bbC$ be the Cohen forcing notion, and $G$ be $(V, \bbC)$-generic.
Then in $V[G]$,
there are regular $T_1$ Lindel\"of spaces $X_0$ and $X_1$
such that $X_0^n$ and $X_1^n$ are Lindel\"of for every $n<\om$,
but the product space $X_0 \times X_1$ has an open cover which has no subcover of size $<\ka$.
\end{prop}
\begin{proof}
Let $X_0$ and $X_1$ be the spaces constructed in the proof of Proposition 3.1 in \cite{U1}.
We sketch the constructions for the convenience of the readers.

We work in $V$. Fix $\la \ge \ka$ such that there is  no $\sigma$-complete fine ultrafilters over $\pkl$.
Let $\mathrm{Fine}(\pkl)$ be the set of all fine ultrafilters over $\pkl$.
Identifying $\pkl$ as a discrete space,
$\mathrm{Fine}(\pkl)$ is a closed subspace of the Stone-\v Cech compactification of $\pkl$,
hence  $\mathrm{Fine}(\pkl)$ is a compact Hausdorff space.
Let $\{\calA_\alpha \mid \alpha<\mu\}$ be an enumeration of all countable partitions of $\pkl$,
and for $\alpha<\mu$, fix an enumeration $\{A^\alpha_n \mid n<\om\}$ of $\calA_\alpha$.
Take a $(V, \bbC)$-generic $G$. 
Let $r=\bigcup G$, which is a function from $\om$ to $\{0,1\}$.
Let $a=\{n<\om \mid r(n)=0\}$, and $b=\{n<\om \mid r(n)=1\}$.

In $V[G]$, we define $X_0$ and $X_1$ as follows.
The underlying set of $X_0$ is $\mathrm{Fine}(\pkl)^V$, the set of all fine ultrafilters over $\pkl$
in $V$.
The topology of $X_0$ is generated by the family
$\{ \{U \in \mathrm{Fine}(\pkl)^V \mid A \in U, A^{\alpha_i}_n \notin U$ for every $i \le k$ and $n \in a\}\mid
A \in V$, $A \subseteq (\pkl)^V$, $\alpha_0,\dotsc, \alpha_k <\mu, k<\om\}$.
The space $X_1$ is defined as a similar way with replacing $a$ by $b$.
$X_0$ and $X_1$ are zero-dimensional regular $T_1$ Lindel\"of spaces in $V[G]$.
Moreover $X_0 \times X_1$ has an open cover which has no subcover of size $<\ka$.
In addition, we can check that $X_0^n$ and $X_1^n$ are Lindel\"of for every $n<\om$
(see the proof of Proposition 3.9 in \cite{U1}).
\end{proof}

For a Tychonoff space $X$,
let $C_p(X)$ be the space of all continuous functions from $X$ to the real line $\bbR$
with the pointwise convergent topology.
For a topological space $X$,
the Lindel\"of degree $L(X)$ is the minimum infinite cardinal $\ka$
such that every open cover of $X$ has a subcover of size $\le \ka$.
Hence $X$ is Lindel\"of if and only if $L(X)=\om$.

\begin{thm}[Arhangel'ski\u \i-Pytkeev \cite{A,P}]\label{cp}
Let $X$ be a Tychonoff space, and $\nu$ a cardinal.
Then $L(X^n) \le \nu$ for every $n<\om$
if and only if $t(C_p(X)) \le \nu$.
In particular,
each finite power of $X$ is Lindel\"of if and only if $C_p(X)$ is countably tight.
\end{thm}

\begin{prop}\label{4.6}
Let $\ka$ be a cardinal which is not $\om_1$-strongly compact.
Let $\bbC$ be the Cohen forcing notion, and $G$ be $(V, \bbC)$-generic.
Then in $V[G]$,
there are regular $T_1$ Lindel\"of spaces $X_0$ and $X_1$
such that $C_p(X_0)$ and $C_p(X_1)$ are countably tight
and  $t(C_p(X_0) \times C_p(X_1)) \ge \ka$.
\end{prop}
\begin{proof}
Let $X_0$ and $X_1$ be spaces from Proposition \ref{4.4}.
By Theorem \ref{cp}, $C_p(X_0)$ and $C_p(X_1)$ are countably tight.
It is clear that $C_p(X_0) \times C_p(X_1)$ is homeomorphic to $C_p(X_0 \oplus X_1)$,
where $X_0 \oplus X_1$ is the topological sum of $X_0$ and $X_1$.
Since $(X_0 \oplus X_1)^2$ has a closed subspace which is homeomorphic to $X_0 \times X_1$,
we have $L((X_0 \oplus X_1)^2) \ge L(X_0 \times X_1)$,
and by Proposition \ref{4.4}, we know $L(X_0 \times X_1) \ge \ka$.
Hence $L((X_0 \oplus X_1)^2) \ge \ka$,
and we have $t(C_p(X_0) \times C_p(X_1)) =t(C_p(X_0 \oplus X_1))\ge \ka$ by Theorem \ref{cp}.
\end{proof}

Combining these results we have the following theorem, which contains Theorem \ref{thm5}:

\begin{thm}
Let $\bbC$ be the Cohen forcing notion, and $G$ be $(V, \bbC)$-generic.
Then for every cardinal $\ka$ 
the following are equivalent  in $V[G]$:
\begin{enumerate}
\item $\ka$ is $\om_1$-strongly compact.
\item For every countably tight spaces $X$ and $Y$
we have  $t(X \times Y) \le \ka$.
\item For every countably tight Tychonoff spaces $X$ and $Y$
we have  $t(X \times Y) \le \ka$.
\item For every regular $T_1$ Lindel\"of spaces $X$ and $Y$,
if $C_p(X)$ and $C_p(Y)$ are countably tight
then $t(C_p(X) \times C_p(Y)) \le \ka$.
\end{enumerate}
\end{thm}
\begin{proof}
Note that for a cardinal $\ka$,
$\ka$ is $\om_1$-strongly compact in $V$ if and only if
so is in $V[G]$ (e.g., see \cite{U1}).

Implications (2) $\Rightarrow$ (3) $\Rightarrow$ (4) are trivial.
(1) $\Rightarrow$ (2) follows from Proposition \ref{prop4.2},
and Proposition \ref{4.6} shows (4) $\Rightarrow$ (1).
\end{proof}

Theorem \ref{thm5} is a consistency result,
and the following natural question arises:
\begin{question}
In ZFC, is the least $\om_1$-strongly compact cardinal
an exact upper bound on 
the tightness of the products of two countably tight spaces?
How about Fr\v echet-Urysohn spaces?
\end{question}

To answer this question for the  Fr\v echet-Urysohn case,
we have to consider other spaces instead of $C_p(X)$, because
if $C_p(X)$ and $C_p(Y)$ are Fr\v echet-Urysohn,
then so is $C_p(X) \times C_p(Y)$. 
This can be verified as follows.
It is known that if $X$ is compact Hausdorff, then $X$ is scattered if and only if
 $C_p(X)$ is Fr\v echet-Urysohn
(Pytkeev \cite{P2}, Gerlits \cite{Ger}). 
In addition the compactness assumption can be weakened to that
each finite power is Lindel\"of.
Gewand \cite{Ge} proved that if $X$ and $Y$ are Lindel\"of and $X$ is scattered, then
$X \times Y$ is Lindel\"of as well.

Finally we present another application of Lemma \ref{4.1}.
Bagaria and da Silva \cite{BS} proved that if $\ka$ is $\om_1$-strongly compact,
$X$ is a first countable space, and every subspace of $X$ with size $<\ka$ is normal,
then $X$ itself is normal.
Using Lemma \ref{4.1}, we can weaken the first countable assumption to the countable tightness assumption.
\begin{prop}
Let $\ka$ be an $\om_1$-strongly compact cardinal,
and $X$ a countably tight topological space.
If every subspace of $X$ with size $<\ka$ is 
normal, then the whole space $X$ is also normal.
\end{prop}
\begin{proof}
Take pairwise disjoint closed subsets $C$ and $D$.
Let $U$ be a $\sigma$-complete fine ultrafilter over $\p_\ka X$.
By the assumption, for each $s \in \p_\ka X$,
the subspace $s$ is normal. Hence we can find open sets $O_s$  and $V_s$ 
such that $s \cap C \subseteq O_s$, $s \cap D \subseteq V_s$, and $O_s \cap V_s \cap s=\emptyset$.
Define $O$ and $V$ by
$x \in O \iff \{s \in \pkl \mid x \in O_s\} \in U$ and
$x \in V \iff \{s \in \pkl \mid x \in V_s\} \in U$.
By using Lemma \ref{4.1}, we know that $O$ and $V$ are open,
and it is easy to see that $O$ and $V$ are disjoint, $C \subseteq O$, and $D \subseteq V$.
\end{proof}

\subsection*{Acknowledgments}
The author would like to thank the referee  for many corrections and valuable comments.
This research was supported by 
JSPS KAKENHI Grant Nos. 18K03403 and 18K03404.

\printindex


\begin{thebibliography}{100}
\bibitem{A} A.~V.~Arhangel'ski\u \i,
\emph{On some topological spaces that arise in functional analysis}, Russian Math. Surveys Vol.~31, No.~5
(1976), 14--30.
\bibitem{BM1} J.~Bagaria, M.~Magidor, \emph{ Group radicals and strongly compact cardinals}.
Trans. Am. Math. Soc. Vol.~366, No.~4 (2014), 1857--1877.
\bibitem{BM2} J.~Bagaria, M.~Magidor, On \emph{$\om_1$-strongly compact cardinals}.
J.~Symb. Logic ~ Vol. 79, No.~1 (2014), 268--278.
\bibitem{BS} J.~Bagaria, S.~G.~da Silva, \emph{$\om_1$-strongly compact cardinals and normality}.
Preprint.
\bibitem{EGKTT} K.~Eda, G.~Gruenhage, P.~Koszmider, K.~Tamano, S.~Todor\v cevi\'c,
\emph{Sequential fans in topology}.
Topol.~App. Vol.~67 (1995), 189--220.
\bibitem{Ger}J.~Gerlits, \emph{Some properties of $C(X)$ II}.
Topol. Appl. Vol.~15(3) (1983) 255--262.
\bibitem{Ge}M.~E.~Gewand, 
\emph{The Lindel\"of degree of scattered spaces and their products.}
J. Aust. Math. Soc., Ser. A 37, (1984), 98--105.

\bibitem{G} I.~Gorelic, \emph{On powers of Lindel\"of spaces}.
Comment. Math. Univ. Carol. Vol. 35, No. 2 (1994), 383--401.
\bibitem{K} J.~Ketonen, \emph{Strong compactness and other cardinal sins}.
Ann. Math. Logic Vol.~5 (1972), 47--76.
\bibitem{P}E.~G.~Pytkeev,
\emph{The tightness of spaces of continuous functions}.
Russian Math. Survey Vol.~37, No.1 (1982), 176--177. 
\bibitem{P2} E.~G.~Pytkeev,
\emph{Sequentiality of spaces of continuous functions}.
Russ. Math. Survery Vol.~37, No. 5 (1982) 190--191.

\bibitem{U1}
T.~Usuba, \emph{$G_\delta$-topology and compact cardinals}.
Fund. Math. Vol. 246 (2019), 71--87.
\bibitem{U2}
T.~Usuba, \emph{A note on the tightness of $G_\delta$-modification}.
Topol. App. Vol.~265, 15 (2019), 106820.
\end{thebibliography}
\end{document}